\documentclass[letterpaper, 10 pt, conference]{ieeeconf}

\IEEEoverridecommandlockouts %
\makeatletter
\newcommand{\removelatexerror}{\let\@latex@error\@gobble}
\makeatother
\usepackage[T1]{fontenc}
\usepackage[utf8]{inputenc}

\usepackage{cite}
\usepackage{dsfont}
\usepackage[super]{nth}

\newcommand{\xvbox}[2]{\makebox[#1][l]{#2}} 

\newcommand{\Pis}[1]{\Pi_{\mathrm{state}}(#1)}
\newcommand{\Pit}[1]{\Pi_{\mathrm{traj}}(#1)}
\newcommand{\Pic}[1]{\Pi_{\mathrm{cost}}(#1)}

\usepackage[english]{babel}
\usepackage[usenames,dvipsnames,table]{xcolor}
\usepackage{amsmath,amsthm,amssymb}
\usepackage{graphicx}
\usepackage{url}
\usepackage{float}

\usepackage{tikz}
\usetikzlibrary{shapes,arrows,positioning,automata}

\usepackage[mathscr]{euscript}
\usepackage{mathtools}

\usepackage{xargs} %
\usepackage{subcaption} %
\usepackage[font=footnotesize,labelfont=footnotesize]{caption}
\usepackage[noend,ruled,shortend,commentsnumbered,linesnumbered]{algorithm2e}
\usepackage{etoolbox} \let\bbordermatrix\bordermatrix
\patchcmd{\bbordermatrix}{8.75}{4.75}{}{}

\newcommand{\real}{\mathbb{R}}
\newcommand{\realnonnegative}{{\mathbb{R}}_{\ge 0}}

\newcommand{\naturalnumbers}{\mathbb{N}}
\newcommand{\norm}[1]{\ensuremath{\| #1 \|}}
\newcommand{\until}[1]{[\, #1 \,]}
\newcommand{\map}[3]{#1:#2 \rightarrow #3}
\newcommand{\setdef}[2]{\{#1 \; | \; #2\}}

\newcommand{\myemphc}[1]{\emph{#1}} 
\newcommand{\setr}[1]{\{#1\}}

\newcommand{\abs}[1]{|#1|}

\newcommand{\xtraj}{\mathsf{x}}
\newcommand{\utraj}{\mathsf{u}}
\newcommand{\Unsafe}{\mathcal{U}\mathcal{I}}

\newcommand{\SSs}{\mathcal{S}}
\newcommand{\SSo}{\overline{\SSs}}

\newcommand{\drmpc}{\mathtt{Safe\_MPC}}
\newcommand{\st}{\operatorname{subject \text{$\, \,$} to}}
\renewcommand{\st}{\operatorname{s.t.}}

\newcommand{\Eb}{\mathbb{E}}
\newcommand{\Pb}{\mathbb{P}}

\newcommand{\Data}{\widehat{\mathcal{W}}}

\newcommand{\Qb}{\mathbb{Q}}

\newcommand{\HH}{\mathcal{H}}
\newcommand{\II}{\mathcal{I}}

\newcommand{\OO}{\mathcal{O}}

\newcommand{\PP}{\mathcal{P}}

\newcommand{\UU}{\mathcal{U}}

\newcommand{\WW}{\mathcal{W}}
\newcommand{\XX}{\mathcal{X}}

\newcommand{\CVaR}{\operatorname{CVaR}}
\newcommand{\VaR}{\operatorname{VaR}}

\newcommand{\Pbhat}{\widehat{\Pb}}

\newcommand{\dist}{\operatorname{dist}}

\newcommand{\what}{\widehat{w}}

\newcommand{\costgo}[2]{J_{(#1:#2)}}

\newcommand{\Pbcen}{\Pb^{\mathrm{cen}}}
\newcommand{\wcen}{w^{\mathrm{cen}}}

\newcommand{\Ncl}{N^{\mathrm{clu}}}
\newcommand{\idmap}{\mathrm{id}}
\newcommand{\Pbclu}{\Pb^{\mathrm{clu}}}

\newcommand{\XXinner}[1]{\XX^{\mathrm{inn},#1}}

\newcommand{\dmin}{d_{\mathrm{min}}}
\newcommand{\binner}[1]{b^{\mathrm{inn},#1}}
\newcommand{\Xappsafe}{\overline{\XX}}
\newcommand{\Xsafeiter}{\XX_{\mathrm{safe}}}
\newcommand{\XWass}{\XX_{\mathrm{Wass}}}
\newcommand{\XcluWass}{\XX_{\mathrm{cl-Wass}}}
\newcommand{\XinnWass}{\XX_{\mathrm{inn-Wass}}}
\newcommand{\Xb}{\mathbb{X}}

\newcommand{\oprocendsymbol}{\hbox{$\bullet$}}
\newcommand{\oprocend}{\relax\ifmmode\else\unskip\hfill\fi\oprocendsymbol}
\newcommand{\longthmtitle}[1]{\mbox{}\textup{\textsl{(#1):}}}

\allowdisplaybreaks

\newcommand{\ifinclude}[1]{}

\renewcommand{\theenumi}{(\roman{enumi})}

\makeatletter
\newcommand{\thickhline}{%
	\noalign {\ifnum 0=`}\fi \hrule height 1pt
	\futurelet \reserved@a \@xhline
}
\newcolumntype{"}{@{\hskip\tabcolsep\vrule width 1pt\hskip\tabcolsep}}
\makeatother

\newtheorem{theorem}{Theorem}[section]
\newtheorem{proposition}{Proposition}[section]
\newtheorem{lemma}[theorem]{Lemma}

\theoremstyle{definition}

\newtheorem{assumption}{Assumption}[section]
\newtheorem{remark}[theorem]{Remark}

\usepackage[normalem]{ulem} %
\definecolor{new}{rgb}{0.55,0,0.55}
\usepackage[prependcaption,colorinlistoftodos]{todonotes}

\renewcommand{\baselinestretch}{.99}

\title{Wasserstein distributionally robust risk-constrained iterative MPC for motion planning: computationally efficient approximations} 
\author{Alireza Zolanvari  \qquad Ashish Cherukuri \thanks{The authors are with the Engineering and Technology Institute Groningen, University of Groningen. Email: \texttt{\{a.zolanvari, a.k.cherukuri\}@rug.nl}. This work was partly supported with a scholarship from the Data Science and Systems Complexity (DSSC) Center, University of Groningen. }}

\begin{document}
	\maketitle
	\thispagestyle{empty}
	\pagestyle{empty}
	
	\begin{abstract}
		This paper considers a risk-constrained motion planning problem and aims to find the solution combining the concepts of iterative model predictive control (MPC) and data-driven distributionally robust (DR) risk-constrained optimization. In the iterative MPC, at each iteration, safe states visited and stored in the previous iterations are imposed as terminal constraints. Furthermore, samples collected during the iteration are used in the subsequent iterations to tune the ambiguity set of the DR constraints employed in the MPC. In this method, the MPC problem becomes computationally burdensome when the iteration number goes high. To overcome this challenge, the emphasis of this paper is to reduce the real-time computational effort using two approximations. First one involves clustering of data at the beginning of each iteration and modifying the ambiguity set for the MPC scheme so that safety guarantees still holds. The second approximation considers determining DR-safe regions at the start of iteration and constraining the state in the MPC scheme to such safe sets. We analyze the computational tractability of these approximations and present a simulation example that considers path planning in the presence of randomly moving obstacle. 
	\end{abstract}
	
	\section{Introduction}\label{sec:intro}
	Real-world autonomous mobile robots usually navigate in unknown or partially known environments. Thus, safety is one of the most significant priorities in motion planning problems. There are different optimal control techniques for providing safety guarantees in such situations, among which robust and probabilistic approaches are popular. Alternatively, risk-aware control design has gathered attention recently due to its ability to tune the conservativeness and the safety level of the designed controller between robust and probabilistic approaches. Moreover, the optimization problems leading to risk-aware decisions are often convex and tractable for a large class of risk measures. One challenge persists for all the above-listed methods, that of not knowing the distribution of the uncertainty fully and only having access to a small number of samples of it. In such scenarios, distributionally robust (DR) decisions provide an elegant way of tuning safety and cost-efficiency. Motivated by this, this paper considers a risk-constrained optimal control problem for a motion planning problem and improves on the iterative DR MPC scheme formulated in~\cite{AZ-AC:22-ecc}. The method presented in~\cite{AZ-AC:22-ecc} uses distributional robustness to find safe trajectories, even when the samples of the uncertainty are few in number. However, the method suffers from heavy computational burden when considering the realistic scenario of continuous distributions. To this end, we provide ``low-complexity'' approximations of the DR risk constraint considered in the MPC routine without compromising on safety. This process brings the iterative DR MPC method closer to being real-time implementable.

	\subsubsection*{Literature review}
	Distributional robustness in risk-constrained model predictive control is explored extensively in~\cite{PC-PP:2022, CM-SL:21, AH-IY:2021-tac, AN-ARH:2022}. While most of them focus on out-of-sample performance guarantees, the proposed schemes become computationally burdensome as the number of available samples grows. To overcome this challenge, two different strategies have been suggested in the distributionally robust (DR) optimization literature. On the one hand,~\cite{DL-SM:2020, FF-PJG:2021, IW-CB-BV-BS:2022} reduce the number of samples by data compression and provide performance guarantees for optimizers of the DR problems formulated by the compressed data. The idea of complexity reduction using data compression is known as scenario reduction and was introduced in~\cite{JD-NG-WR:2003}. On the other hand, constraint tightening approaches, proposed in~\cite{CM-SL:2020, AD-MA-JWB:2022-DRMPC, MF-JL:2022}, tackle the computational challenge by inner-approximating the distributionally robust region and limiting the feasible states to the DR-safe region. None of the above-mentioned studies explore these techniques for iterative approaches. Iterative approaches are well-suited to many real-world problems as they can overcome the challenge of insufficient data points through progressive exploration of the environment and repeated execution of the task. With this motivation, we explore both scenario reduction and constraint-tightening types of approximations of the iterative algorithm from~\cite{AZ-AC:22-ecc}. Our emphasis is on reducing the computational effort and making our method real-time implementable. 
	\subsubsection*{Setup and Contributions} 
	We start in Section~\ref{sec:problem} with the explanation of the risk-constrained optimal control problem that is tailored for the motion planning problem. Here the risk constraint encodes collision avoidance in presence of a random obstacle. In Section~\ref{sec:algo}, we present the iterative DR MPC method from~\cite{AZ-AC:22-ecc} as a solution strategy and explain the general framework when considering Wasserstein ambiguity sets. We present in Section~\ref{sec:approx} the finite-dimensional reformulation of the DR risk constraints when considering Wasserstein ambiguity sets. We discuss the computational challenge of imposing these reformulated constraints, which lead to our main contributions that entail providing two computationally efficient approximations of the DR risk constraints. The first one considers Wasserstein ambiguity 
	sets with a larger radius and the center as the distribution generated using clustered data. The second one generates an inner estimate of the set of points that satisfy DR risk constraints. The inner estimate has a simple form of being the union of half-spaces. We show that both our approximations are safe in the sense that they satisfy the originally imposed risk constraint up to a pre-specified probability. We identify conditions under which the iterative DR MPC with the defined approximations are recursively feasible and asymptotically convergent. Lastly, we demonstrate the advantage of our method for a risk-averse path planning task.

	\section{Preliminaries}\label{sec:prelims}
	Here we collect notation and mathematical background.  
	\subsubsection{Notation}\label{subsec:notation}
	Let $\real$, $\realnonnegative$, and $\naturalnumbers_{\ge 1}$ denote the set of real, nonnegative real, and natural numbers excluding zero, respectively. 
	Let $\norm{\cdot}$ denote the $2$-norm. For $N \in \naturalnumbers_{\ge 1}$, we denote $[N] := \{1,\dots,N\}$ and $[N]_0 := \{0,1,\dots,N\}$. Given $x \in \real$, we let $[x]_+ = \max(x,0)$. 
	
	\subsubsection{Conditional Value-at-Risk}\label{subsec:cvar}
	We review notions on conditional value-at-risk (CVaR) from~\cite{AS-DD-AR:14}. Given a real-valued random variable $Z$ with probability distribution $\Pb$ and $\beta \in (0,1)$, the \myemphc{value-at-risk} of $Z$ at level $\beta$, denoted $\VaR_\beta^{\Pb}[Z]$, is the left-side $(1-\beta)$-quantile of $Z$. Formally,
	\begin{align*}
		\VaR_\beta^\Pb [Z]  	& = \inf \setdef{\zeta}{\Pb(Z \le \zeta) \ge 1-\beta}.
	\end{align*}
	The \myemphc{conditional value-at-risk (CVaR)} of $Z$ at level $\beta$, denoted $\CVaR_\beta^\Pb [Z]$, is given as 
	\begin{align}\label{eq:cvar-def-alt}
		\CVaR_\beta^\Pb [Z] = \inf_{t \in \real} \Bigl\{ t + \beta^{-1} \Eb^\Pb[Z - t]_+ \Bigr\},
	\end{align}
	where $\Eb^\Pb[\,\cdot\,]$ denotes expectation under $\Pb$. Under continuity of the CDF of $Z$, we have 
	$\CVaR_\beta^\Pb [Z] := \Eb^\Pb[Z \ge \VaR_\beta^\Pb [Z]]$.
	The parameter $\beta$ characterizes the risk-averseness. When $\beta$ is close to unity, the decision-maker is risk-neutral, whereas, $\beta$ close to the origin implies high risk-averseness. 
	\subsubsection{Wasserstein metric}\label{sec:Wass}
	Given a compact set $\WW \subset \real^{n_w}$, let $\PP(\WW)$ be the set of Borel probability
	measures supported on $\WW$. Following~\cite{PME-DK:18}, the $1$-Wasserstein metric between measures $\mu, \nu \in \PP(\WW)$ is
	\begin{equation}\label{eq:def_wasserstein}
		d_W(\mu,\nu) := \min_{\gamma \in \HH(\mu,\nu)}
		\left\{\int_{\WW \times \WW} \norm{w_1-w_2} \gamma(dw_1,dw_2) \right\},
	\end{equation}
	where $\HH(\mu,\nu)$ is the set of all distributions on
	$\WW \times \WW$ with marginals $\mu$ and $\nu$. %
	
	\section{Problem Statement} \label{sec:problem}
	
	Consider the following discrete-time system:
	\begin{equation}\label{sys}
		x_{t+1}  = f(x_t,u_t), 
\end{equation}
where $\map{f}{\real^{n_x} \times \real^{n_u}}{\real^{n_x}}$ represents the dynamics and $x_t\in\real^{n_x}$ and $u_t\in\real^{n_u}$ are the state and control input at time $t$, respectively. The system state and control input are required to satisfy the following deterministic constraints:
\begin{equation}
	x_t\in \XX , \quad u_t\in \UU , \quad  \forall t\geq 0,
\end{equation}
where $\XX \subset \real^{n_x}$ and $\UU \subset \real^{n_u}$ are assumed to be compact convex sets. We assume without loss of generality that $0 \in \UU$. Our objective in the motion planning problem is to drive the system from an initial state to a target equilibrium point $x_F \in \XX$ while ensuring a suitable safety requirement. We encode this task as the following infinite-horizon risk-constrained optimal control problem:
\begin{subequations}\label{eq:IHOCP}
	\begin{align}
		\min\quad  &\sum_{t=0}^{\infty}r(x_{t}, u_{t})\label{eq:IHOCP-obj} %
		\\
		\text{s.t.}  \quad & x_{t+1} = f(x_t,u_t), \quad \forall t\geq 0,\label{eq:IHOCP-a}
		\\
		\quad & x_t \in \XX , u_t\in \UU , \quad  \forall t\geq 0,\label{eq:IHOCP-c}
		\\
		\quad & x_0 = x_S,\label{eq:IHOCP-b}
		\\
		\quad & \CVaR_{\beta}^\mathbb{P}\left[g(x_t,w)\right] \leq 0, \quad \forall t \geq 0. \label{eq:IHOCP-d}
	\end{align}
\end{subequations}
Here, $x_S \in \XX$ is the initial state and the stage-cost $\map{r}{\XX \times \UU}{\realnonnegative}$ is assumed to be continuous satisfying $r(x,u) = 0$ if and only if $(x,u) = (x_F,0)$. Further, the constraint~\eqref{eq:IHOCP-d} represents the safety guarantee, where $\CVaR$ stands for the conditional value-at-risk (see Section~\ref{subsec:cvar} for details), $w$ is a random variable with distribution $\Pb$ supported on the polyhedral convex compact set $\WW := \setdef{w\in\real^{n_w}}{Hw\le h}$, the parameter $\beta > 0$ is the risk averseness coefficient, and the continuous function $\map{g}{\XX \times \WW}{\real}$  is referred to as the constraint function. Next, we will make $g$ more precise for the case of avoiding polyhedral obstacles. 

Let $C \in \real^{n_p \times n_x}$ be a matrix that gives the components corresponding to the position of a state $x \in \real^{n_x}$ as $Cx$. Consequently, the feasible region in the position coordinates is denoted by $C \XX$. We assume the presence of one uncertain obstacle in $C \XX$. At the unperturbed position, the space occupied by the obstacle is defined by
	$\mathcal{O} : = \setdef{p \in \real^{n_p}}{Ap \le b}$,
where $A \in \real^{M \times n_p}$, $b \in \real^M$, and $\OO$ is assumed to be compact.  As the position of the obstacle is uncertain we assume that given a realization $w \in \WW \subset \real^{n_p}$ of the random variable, the occupancy of the obstacle is given as
\begin{align*}
	\OO_w := \OO + w = \setdef{p+w}{Ap \le b}.
\end{align*}
We assume that $\OO_w \subset C \XX$ for all realizations $w \in \WW$. %
We say that the state $x$ is safe in terms of collision with $\OO_w$ if the following constraint is met
\begin{align}\label{eq:g-func}
	g(x,w) := d_{\mathrm{min}} - \dist(Cx,\OO_w) \le 0,
\end{align}
where $\dist(Cx,\OO_w)$ stands for the Euclidean distance of the point $Cx$ from the set $\OO_w$. Specifically,
\begin{align*}
	\dist(Cx,\OO_w) = \min_{y \in \OO_w} \norm{Cx - y}.
\end{align*}
Note that using the structure of the $\OO_w$, the distance can be equivalently written as
\begin{align*}
	\dist(Cx, \OO_w) := \max_{m\in[M]}\left[A_m(Cx-w)-b_m\right]_+,
\end{align*}
where $A_m$ and $b_m$ are the $m$-th row of $A$ and $m$-th component of $b$, respectively. For the sake of simplicity and without loss of generality, we assume that $\norm{A_m}=1$ for all $m$. 
Throughout the paper, we assume $g$ of the above form. In lieu of the above definition, we call $x$ to be $\Pb$-risk safe if
\begin{align*}
	\CVaR_\beta^\Pb[g(x,w)] = \CVaR_\beta^\Pb[ d_{\mathrm{min}} - \dist(Cx,\OO_w)]\le 0.
\end{align*}

We approach solving the infinite-horizon problem~\eqref{eq:IHOCP} using approximations of the iterative MPC scheme proposed in~\cite{AZ-AC:22-ecc}. The method given in~\cite{AZ-AC:22-ecc} combines ideas from iterative learning MPC~\cite{UR-FB:17-tac} and distributionally robust (DR) risk-constrained optimization problems~\cite{AC-ARH:2020}. However, in doing so, the optimization problem at the core of the MPC problem becomes computationally cumbersome, especially when the iteration count goes high, the underlying distribution is continuous, and the ambiguity set is defined using the Wasserstein metric. Specifically, in this case, the optimization problem has mixed-integer decision variables, where the number of constraints and the number of decision variables grow linearly with the number of samples. The objective of this paper is to ease this computational burden by rendering the size of the optimization problem solved in the MPC independent of the number of samples while maintaining safety guarantees. %

\section{DR-based safety-constrained iterative MPC}\label{sec:algo}
Here we present our general iterative framework that is borrowed from~\cite{AZ-AC:22-ecc}. We present this scheme in a form where the safety constraints embedded in the MPC problem are abstract. In the following section, we specify three ways of generating these safety constraints, all of which form different subsets of the set of points that satisfy the risk constraint~\eqref{eq:IHOCP-d} in a distributionally robust manner. We then discuss their computational and statistical properties.

\subsection{Building blocks of the iterative scheme}
An iteration involves generating a trajectory of the system~\eqref{sys}. In particular, the trajectory of the $j$-th iteration is denoted as 
\begin{equation}\label{cl-traj}
	\begin{split}
		\xtraj^j& :=[x_0^j, x_1^j, \dots, x_t^j, \dots, x_{T_j}^j],\\
		\utraj^j& :=[u_0^j, u_1^j, \dots, u_t^j, \dots, u_{T_j - 1}^j],
	\end{split}
\end{equation}
where $x_t^j$ and $u_t^j$ are the system state and the control input at time $t$, respectively. We assume that $x_0^j = x_S$ for all $j \geq 1$ and that each trajectory $j$ consists of a finite number of time steps $T_j \in \naturalnumbers_{\ge1}$. We also refer to this quantity as the length of the trajectory. 

Recall that the distribution $\Pb$ of the random variable $w$ is supported on a polyhedral convex compact set $\WW := \setdef{w\in\real^{n_w}}{Hw\le h}$. At the start of iteration $j$, we assume that we have $N_{j-1} \in \naturalnumbers_{\ge 1}$ number of samples of the uncertainty available to us. We denote this dataset as $\Data^{j-1} := \setr{\what_1, \dots, \what_{N_{j-1}}}$. We assume that the system gathers $M_j \in \naturalnumbers_{\ge 1}$ number of samples of the uncertainty during the $j$-th iteration. Thus, $N_j = N_{j-1} + M_j$. If one sample is gathered at each time step of the iteration, then $M_j = T_j$. The data could either be drawn from the distribution $\Pb$ in an i.i.d. manner or could be obtained from other distributions that are close to $\Pb$ in some appropriate metric. In both cases, our approach of enforcing risk constraint~\eqref{eq:IHOCP-d} for all distributions in an appropriately defined set helps ensure the system's safety, even when the number of available samples is low. To this end, we impose an assumption on the data-gathering process. First, we define the empirical distribution corresponding to the dataset $\Data^{j-1}$ as 
\begin{align*}
	\Pbhat_{N_{j-1}} : = \textstyle \frac{1}{N_{j-1}} \sum_{i=1}^{N_{j-1}} \delta_{\what_i},
\end{align*}
where $\delta_{\what_i}$ is the dirac-delta distribution placed at the point $\what_i$. Given this distribution and a radius $\theta_{j-1} > 0$, we construct the Wasserstein ambiguity set as 
\begin{align}
	\mathcal{B}(\Pbhat_{N_{j \! - \! 1}},\theta_{j \! - \! 1}) \! := \! \setdef{\Qb \! \in \!  \PP(\WW) \! }{ \! d_W(\Qb,\Pbhat_{N_{j-1}}) \! \le \!  \theta_{j-1}}, \label{eq:B-Wass}
\end{align} 
where $d_W$ is the Wasserstein metric (see Section~\ref{sec:Wass} for the definition) and $\PP(\WW)$ are all distributions supported on $\WW$. We then assume:
\begin{assumption}\longthmtitle{Confidence guarantee of $\Pb$ contained in the ambiguity set}\label{as:conf}
	For any iteration $j$, we are given a radius $\theta_{j-1} > 0$ such that for any dataset $\Data^{j-1}$ gathered by the end of the $(j-1)$-th iteration, we have
	\begin{align}\label{eq:conf-P-in-ambiguity}
		\operatorname{Prob}[\Pb \in \mathcal{B}(\Pbhat_{N_{j-1}},\theta_{j-1})] \ge \zeta,
	\end{align}
	where $\zeta \in (0,1)$ is a pre-specified confidence level. \oprocend
\end{assumption} 
Note that if samples are drawn i.i.d., then $\operatorname{Prob} = \Pb^{N_{j-1}}$ in the above definition, where $\Pb^{N_{j-1}}$ represents the $N_{j-1}$-fold product of the underlying distribution $\Pb$. Then, one can derive the relationship between  $\zeta$,  $\theta_{j-1}$, and  $N_{j-1}$ such that~\eqref{eq:conf-P-in-ambiguity} holds, see e.g.,~\cite[Theorem 3.4]{PME-DK:18}. Bearing the above definition in mind, a point $x' \in \XX$ satisfies the risk constraint~\eqref{eq:IHOCP-d} with probability $\zeta$ if we ensure that 
\begin{align}\label{eq:XWass}
	x' \! \in & \XX^{j-1}_{\mathrm{DR}} \! :=  \! \setdef{x \in \XX}{\! \! \sup_{\Qb \in \mathcal{B}(\Pbhat_{N_{j \! - \! 1}} \!,\theta_{j \! - \! 1})} \! \! \CVaR_{\beta}^\Qb [g(x,w)] \! \le \! 0}.
\end{align} 
The aim of our algorithm is to seek such trajectories, those that satisfy the risk constraint with probability $\zeta$.  One way is to impose~\eqref{eq:XWass}  in our MPC routine. However, the resulting optimization comes with a significant computational burden. Thus, in Section~\ref{sec:approx} we define approximations of the set given in~\eqref{eq:XWass} with the purpose of balancing computational ease and optimality while ensuring safety throughout.

We now describe other key elements of the iterative learning MPC. Given the trajectory $(\xtraj^j,\utraj^j)$ generated in iteration $j$, the cost-to-go at time $t$ is denoted as:
\begin{align}\label{eq:to-go}
	\costgo{t}{\infty}^j :=  \textstyle \sum_{k=t}^\infty r(x_k^j, u_k^j).
\end{align}
Thus, the cost of the $j$-th iteration is $\costgo{0}{\infty}^j$. Since we assume that the $j$-th trajectory has a finite length $T_j$, for every time step $t \ge T_j$, we assume that the system remains at $x_F$ and the control input is zero. Thus, the infinite sum in~\eqref{eq:to-go} is well-defined as $r(x_F,0) = 0$. In our iterative method, information about the system and the environment grows as iterations progress. The latter is owing to the fact that more data regarding the uncertainty becomes available in each iteration. On the other hand, the former is acquired by means of exploring the state space incrementally. To this end, our method maintains a set of safe states  (along with the minimum cost that it takes to go to the target from them) that were explored in the previous iterations and uses them in an iteration as terminal constraints in the MPC scheme (as proposed in~\cite{UR-FB:17-tac}). Specifically, the sampled safe set obtained at the end of iteration $j$ and to be used in iteration $j+1$, denoted $\SSs^{j} \subseteq \until{j} \times \XX \times \realnonnegative$, is defined recursively as
\begin{align}\label{eq:SS-update-gen}
	\SSs^{j} = \mathbb{S}^j \Bigl(\SSs^{j-1} \cup \setr{(j,x_t^{j},\costgo{t}{\infty}^{j})}_{t=1}^{T_j} \Bigr),
\end{align}
where the set $\setr{(j,x_t^{j},\costgo{t}{\infty}^{j})}_{t=1}^{T_j}$ collects the set of states visited in iteration $j$, along with the associated cost-to-go. The counter $j$ is maintained in this set to identify the iteration to which a state with a particular cost-to-go is associated with. The set $\SSs^{j-1}$ is the set used in iteration $j$. The map $\mathbb{S}^j$ only keeps the states that are safe with respect to the new dataset $\Data^j$. The exact map $\mathbb{S}^j$ is explained in our algorithm.

For ease of exposition, we define maps $\Pit{\cdot}$, $\Pis{\cdot}$, and $\Pic{\cdot}$, such that, given a safe set $\SSs$,  $\Pit{\SSs}$, $\Pis{\SSs}$, and $\Pic{\SSs}$ return the set of all trajectory indices, states, and cost-to-go values that appear in $\SSs$, respectively.  The following assumption is required to initialize our iterative procedure with a nonempty sampled safe set. 
\begin{assumption}\longthmtitle{Initialization with robust trajectory}\label{assump:safeset}
	The first iteration starts with the sampled safe set $\SSs^0$ containing a  finite-length robustly safe trajectory $\xtraj^0$ that starts from $x_S$ and reaches $x_F$. This means that the trajectory $\xtraj^0$ in $\SSs^0$ robustly satisfies all constraints of problem~\eqref{eq:IHOCP}, that is, %
	$x\in\XX,\,g(x, w) \leq 0$ for all $w \in \WW$,
	and all $x \in \Pis{\SSs^0}$.\oprocend
\end{assumption}

Note that for every state stored in the sampled safe set $\SSs^{j}$, we store the cost-to-go from it. However, it is possible that a state appears in multiple trajectories and given that $\SSs^{j}$ contains several trajectories, it is beneficial to maintain a minimum cost-to-go from every state in it. That being the case, we define the map
\begin{align}\label{eq:Qj}
	Q^j(x):=\begin{cases}
		\min\limits_{J\in F^j(x)}J, &\quad x\in\Pis{\SSs^j},
		\\
		+\infty, &\quad x\notin\Pis{\SSs^j},
	\end{cases}
\end{align}
where
\begin{align}\label{def:Fj}
	F^j(x) = \setdef{\costgo{t}{\infty}^i }{ \Pis{\left\{(i, x^i_t, \costgo{t}{\infty}^i)\right\}} = \{x\},&\nonumber
		\\ 
		(i, x^i_t, \costgo{t}{\infty}^i) \in \SSs^j}&.
\end{align}
Here, the set $F^j(x)$ contains all cost-to-go values associated with the state $x \in \Pis{\SSs^j}$ and consequently, the function $Q^j$ determines the minimum among these. 

Given the above-described elements, we now present the optimization problem that lies at the core of our method. For generality, we write the problem for generic current state $x$, sampled safe set $\SSo$, and safety constraint $x \in \Xappsafe \subset \XX$. Let $K \in \naturalnumbers_{\ge 1}$ be the length of the horizon and consider
\begin{equation}\label{eq:DR-RLMPC:main}
	\mathcal{J}_{(\SSo, \Xappsafe)} (x) := \begin{cases}
		\min & \, \,  \sum_{k=0}^{K-1}r(x_{k}, u_{k}) +\overline{Q}(x_{K}) 
		\\
		\st & \, \, x_{k+1} = f(x_{k},u_{k}),   \forall k\in[K-1]_0,
		\\
		& \, \, x_{k} \in \Xappsafe, u_{k}\in\UU,   \forall k\in[K-1]_0,
		\\
		& \, \, x_{0}=x,
		\\
		& \, \, x_{K}\in\Pis{\overline{\SSs}}, 
	\end{cases}
\end{equation}
where $\map{\overline{Q}}{\XX}{\real}$ gives the minimum cost-to-go for all states in $\SSo$ and is calculated in a similar manner as in~\eqref{eq:Qj}. The decision variables in the above problem are $(x_0, x_1, \dots, x_K)$ and $(u_0, u_1, \dots, u_{K-1})$. The set $\SSo$ defines the terminal constraint $x_K \in \Pis{\SSo}$. Finally, the constraint $x_k \in \Xappsafe$ encodes the safety guarantee. 
For iteration $j$ and time step $t$, the MPC scheme solves the finite-horizon problem~\eqref{eq:DR-RLMPC:main} with $x = x_t^j$, $\SSo = \SSs^{j-1}$, and $\overline{Q} = Q^{j-1}$, while $\Xappsafe$ takes one of the following values: $\XWass^{j-1}$, $\XcluWass^{j-1}$, or $\XinnWass^{j-1}$. The set $\XWass^{j-1}$ is defined by the reformulation of the Wasserstein DR risk-constraint, the set $\XinnWass^{j-1}$ is an inner approximation of $\XWass^{j-1}$, and finally,  $\XcluWass^{j-1}$ is defined in a similar way as $\XWass^{j-1}$ but with clustered data. We explain these further in Section~\ref{sec:approx}.

\subsection{Algorithm describing the iterative scheme}
Here we put together the building blocks of our method that were outlined above. The resulting scheme is given in Algorithm~\ref{ag:DR-iteration} and is similar to the algorithm from our previous work~\cite{AZ-AC:22-ecc}. 
Each iteration $j \ge 1$ of Algorithm~\ref{ag:DR-iteration} starts with a sampled safe set $\SSs^{j-1}$ and a safe set $\Xsafeiter^{j-1}$, where the latter is determined using the dataset $\Data^{j-1}$ and the radius $\theta_{j-1}$. We represent this association via the map $\mathbb{X}$ (Line~\ref{ln:ambiguity}). As explained above, the set $\Xsafeiter^{j-1}$ takes value as $\XWass^{j-1}$, $\XcluWass^{j-1}$, or $\XinnWass^{j-1}$. The precise definition of these sets are given in Section~\ref{sec:approx}. Given $\SSs^{j-1}$ and $\Xsafeiter^{j-1}$, the first step of the iteration (Line~\ref{ln:drmpc}) involves generating a trajectory $(\xtraj^j,\utraj^j)$ using $\mathtt{Safe\_MPC}$ routine (described in Algorithm~\ref{ag:DR_MPC}). The dataset is updated to $\Data^j$ and the set $\Xsafeiter^j$ is computed for the next iteration using $\Data^j$ and the radius $\theta_j$ in Line~\ref{ln:ambiguity}. The trajectory $\xtraj^j$ along with its associated cost-to-go is appended to the sampled safe set in Line~\ref{ln:safe}. The set $\Unsafe^j$ collects in Line~\ref{ln:uns-traj} all previous trajectories for which one of the states is not safe with respect to the newly determined set $\Xsafeiter^j$. In Line~\ref{ln:M}, the set of trajectories in $\II^{j-1} \cup \{j\}$ that are not in $\Unsafe^j$ are collected in the set $\II^j$. Consequently, the states visited in trajectories in $\II^j$ are stored in $\SSs^j$ in Line~\ref{ln:ss-update} and their minimum cost-to-go is computed in Line~\ref{ln:Q}. Collectively, Lines~\ref{ln:uns-traj} to~\ref{ln:ss-update} represent the map $\mathbb{S}$ defined in~\eqref{eq:SS-update-gen}. 

\begin{algorithm}[htb]
	\SetAlgoLined
	\DontPrintSemicolon
	\SetKwInOut{Input}{Input}
	\SetKwInOut{Output}{Output}
	\SetKwInOut{init}{Initialize}
	\SetKwInOut{giv}{Data}
	\SetKwInOut{params}{Parameter}
	\Input{%
		\xvbox{2mm}{$\SSs^0$}\quad--$\;$Initial sampled safe set \\
		\xvbox{2mm}{$\Data^0$}\quad--$\;$Initial set of samples \\
		\xvbox{2mm}{$\II^0$}\quad--$\;$Index of trajectory in $\SSs^0$
		\\
	}
	\init{
		$j \gets 1$, $\Xsafeiter^{0}$, $\Unsafe^{0}\gets\emptyset$}
	\While{$j > 0$ }
	{%
		Set $(\xtraj^{j}, \utraj^{j})\gets \mathtt{Safe\_MPC}(\SSs^{j-1}, \Xsafeiter^{j-1})$; $T^j\gets\mathtt{length}(\xtraj^{j})$; $\Data^{j} \gets \Data^{j-1} \cup \{\what_i\}_{i=1}^{M_j}$ \label{ln:drmpc} \;
		Set $\Xsafeiter^j \gets \Xb(\Data^j,\theta_j)$ \label{ln:ambiguity}\; 
		Set $\SSs^{j-1}\gets\SSs^{j-1} \cup \{(j, x_t^j, J^j_{(t:\infty)})\}_{t=1}^{T_j}$ \label{ln:safe}\;
		Set $\Unsafe^j \gets \{ i \in (\II^{j-1}\cup \{j\} ) \, |(i, x, J) \in \SSs^{j-1} \text{ and } x \not \in \Xsafeiter^j\}$ \label{ln:uns-traj} 
		\;
		Set $\II^{j} \gets (\II^{j-1}\cup \{j\} ) \setminus \Unsafe^{j}$ \label{ln:M} \;
		Set $\SSs^{j} \gets \setdef{(i, x, J) \in \SSs^{j-1}}{i \in \II^{j}}$ \label{ln:ss-update}\;
		Compute $Q^j(x)$ for all $x\in\Pis{\SSs^j}$ using~\eqref{eq:Qj}\label{ln:Q}\;
		Set $j\gets j+1$
	}
	
	\caption{Iterative MPC with DR-based safety constraints} %
\label{ag:DR-iteration} 
\end{algorithm}

Algorithm~\ref{ag:DR-iteration} calls the $\drmpc$ routine in each iteration. This procedure is given in Algorithm~\ref{ag:DR_MPC} where at each time step $t$, the finite-horizon problem~\eqref{eq:DR-RLMPC:main} is solved with $x = x_t$. The optimal solution is denoted as 

\begin{equation}\label{eq:optSol}
	\begin{split}
		x_{\mathrm{vec},t}^{*} &= [x_{t|t}^{*}, \dots , x_{t+K|t}^{*}], 
		\\
		u_{\mathrm{vec},t}^{*} &= [u_{t|t}^{*}, \dots , u_{t+K-1|t}^{*}],
	\end{split}
\end{equation}
where $x_{t+k|t}$ is the prediction made at time $t$ regarding the state at time $t+k$. The control at time $t$ is set as the first element $u_{t|t}^{*}$ (Line~\ref{ln:control}) and it is appended to the trajectory $\utraj$. The state is updated and added to $\xtraj$ in Line~\ref{ln:state}. The procedure moves to the next time step with the updated state as $x_{t+1}$.

\begin{algorithm}[htb]
	\SetAlgoLined
	\DontPrintSemicolon
	\SetKwInOut{Input}{Input}
	\SetKwInOut{Output}{Output}
	\SetKwInOut{init}{Initialize}
	\SetKwInOut{giv}{Data}
	\SetKwInOut{params}{Parameter}
	\SetKwProg{Fn}{Function}{:}{}
	\SetKwFunction{FMain}{$\mathtt{Safe\_MPC}$}
	\Fn{\FMain{$\SSo, \overline{\XX}$}}{
		\init{%
			$t\gets0$; $x_0\gets x_S$; $\xtraj \gets[x_0]$, $\utraj \gets[\,\,]$
		}
		Set $\overline{Q}$ as minimum cost-to-go in $\SSo$ (use~\eqref{eq:Qj})\;
		\While{$x_t\neq x_F$}{
			Solve~\eqref{eq:DR-RLMPC:main} with $x=x_t$ and obtain optimal solutions $x_{\mathrm{vec},t}^{*}$ and $u_{\mathrm{vec},t}^{*}$\; \label{step:finite-horizon}
			Set $u_t\gets u^{*}_{t|t}$; $\utraj \gets[u, u_t]$ \label{ln:control}\;
			Set $x_{t+1} \! \gets  \! f(x_t,u_t) $; $\xtraj \! \gets \! [x, x_{t+1}]$; $t \! \gets  \! t+1$ \label{ln:state}\;
		}
		\textbf{return} $(\xtraj, \utraj)$
	}
	\textbf{end}
	\caption{Safe MPC function}
	\label{ag:DR_MPC}
\end{algorithm}

\begin{remark}\longthmtitle{Comparison with~\cite{AZ-AC:22-ecc}}
	We note that the general structure of our algorithm is similar to that in~\cite{AZ-AC:22-ecc}. However, there are two key differences that highlight the contribution of this paper. First, the algorithm in~\cite{AZ-AC:22-ecc} is written in a general form for any type of ambiguity set generated using data. There is no emphasis on computational tractability and the simulations only focus on discrete distributions. On the other hand, we here specify the Wasserstein ambiguity set and focus more on the computational issues. Second, our algorithm and the approximations presented in Section~\ref{sec:approx} are tailored for the motion planning problem; they exploit the structure of the obstacle avoidance constraint to derive fast MPC routines that can be implemented in real-time.
	\oprocend
\end{remark}

\section{Approximations of DR-based safety constraint}\label{sec:approx}

In this section, we explain the three safety sets that are used as constraints in the MPC routine of Algorithm~\ref{ag:DR-iteration}. All these sets lead to trajectories that satisfy the risk constraint~\eqref{eq:IHOCP-d} with probability $\zeta$. Before proceeding further we first provide a finite-dimensional representation of an upper bound of the worst-case CVaR over a Wasserstein ambiguity set. This will depict the computational issues of solving the finite horizon problem~\eqref{eq:DR-RLMPC:main} with two DR-based safety sets $\XWass^{j-1}$ and $\XcluWass^{j-1}$. Consequently, it will also motivate the design of the third safety set $\XinnWass^{j-1}$.

\newcommand{\BB}{\mathcal{B}}
\newcommand{\thetacen}{\theta_{\mathrm{cen}}}

\begin{lemma}\longthmtitle{Reformulation of the worst-case risk}\label{le:reform}
	Consider an atomic distribution 
	\begin{align*}
		\Pbcen : = \sum_{\ell = 1}^L p_\ell \delta_{\wcen_\ell},
	\end{align*}
	where $\wcen_\ell \in \WW$ for all $\ell \in [L]$ and $p_\ell \in (0,1)$ for all $\ell \in [L]$ with $\sum_{\ell = 1}^L p_\ell = 1$. Consider the ambiguity set
	\begin{align*}
		\BB(\Pbcen,\thetacen):= \setdef{\Qb \in \PP(\WW)}{d_W(\Qb,\Pbcen) \le \thetacen},
	\end{align*}
	where $d_W$ is the Wasserstein metric~\eqref{eq:def_wasserstein} and $\PP(\WW)$ are all distributions supported on $\WW$. Then, for the function $g$ given in~\eqref{eq:g-func}, we have 
	\begin{align}
		& \sup_{\Qb \in \BB(\Pbcen,\thetacen)} \CVaR_\beta^\Qb [g(x,w)]  \label{eq:worst-case}
		\\
		& \quad \le \begin{cases} \nonumber\inf  & \lambda \theta_{\mathrm{cen}} - \beta\eta + \sum_{\ell = 1}^L s_\ell
			\\
			\st   
			& p_\ell^{-1}s_\ell - \eta \ge \dmin - \big([ACx-b]^\top\nu
			\\
			& \quad  \; -[(A^\top\!\nu\!-\! H^\top\!\gamma_\ell)^\top\! \wcen_\ell\!+\!\gamma_\ell^\top\! h]\big),  \forall \ell\in [L],
			\\
			& \norm{A^\top\nu\!-\! H^\top\gamma_\ell}\!\le\!\lambda, \, \norm{A^\top\nu}\!\le\!1, \quad \forall \ell\in[L],
			\\
			& \nu, \gamma_\ell\in\realnonnegative^M, \;\lambda, s_\ell\in \realnonnegative, \;\eta \in \real,\; \forall \ell\in[L].  %
		\end{cases}
	\end{align}
\end{lemma}
The proof follows from a similar result in~\cite{AN-ARH:2022} and hence is omitted. %
Note that the above result is written for a general ambiguity set defined with the center $\Pbcen$ and radius $\thetacen$. Such a choice is motivated by the fact that in defining the safety set $\XcluWass^{j-1}$, we will use the distribution derived from clustered data instead of the empirical one.

\subsection{Safety set $\XWass^{j-1}$}
For obtaining the safety set $\XWass^{j-1}$, we replace $\Pbcen$ and $\thetacen$ in Lemma~\ref{le:reform} with the empirical distribution $\Pbhat_{N_{j-1}}$ and $\theta_{j-1}$, respectively. Then, we define $\XWass^{j-1}$ as the set of $x$ for which the upper bound on the worst-case risk given in~\eqref{eq:worst-case} is nonnegative. That is, in this case 
\begin{align}
	\Xb&(\Data^{j-1},\theta_{j-1}) = \XWass^{j-1} \label{eq:safe-1}
	\\
	&: = \Big\{x \in \XX | \exists \nu, \gamma_\ell\in\realnonnegative^M, \;\lambda, s_\ell\in \realnonnegative, \;\eta \in \real,\; \notag
		 \\
	& \qquad  \quad \text{such that } \notag 
	\\
 	& \qquad  \textstyle \lambda \theta_{j-1} - \beta\eta + \sum_{\ell = 1}^{N_{j-1}} s_\ell \le 0, \notag 
 \\
	& \qquad   p_\ell^{-1}s_\ell - \eta \ge \dmin - \big([ACx-b]^\top\nu \notag 
	\\
	& \qquad   \, \, -[(A^\top\!\nu\!-\! H^\top\!\gamma_\ell)^\top\! \what_\ell\!+\!\gamma_\ell^\top\! h]\big), \forall \ell\in [N_{j-1}], \notag 
\\
	&\qquad   \norm{A^\top\nu\!-\! H^\top\gamma_\ell}\!\le\!\lambda, \, \norm{A^\top\nu}\!\le\!1,  \forall \ell\in[N_{j-1}] \Big\}. \notag
\end{align}
The above definition and the inequality~\eqref{eq:worst-case} imply that if $x \in \XWass^{j-1}$, then $\sup_{\Qb \in \mathcal{B}(\Pbhat_{N_{j-1}},\theta_{j-1})} \CVaR_{\beta}^\Qb[g(x,w)] \le 0$. Hence, when $\Xsafeiter^{j-1} = \XWass^{j-1}$ in Line~\ref{ln:ambiguity} of Algorithm~\ref{ag:DR-iteration}, then the generated trajectory in the $j$-th iteration is $\Pb$-risk safe with probability $\zeta$. 

Notice that in one of the constraints on the right-hand side of~\eqref{eq:safe-1}, the decision variables $x$ and $\nu$ appear in a bilinear term. Hence, enforcing the state to be in $\XWass^{j-1}$ renders the optimization problem  nonconvex with the number of constraints and the decision variables scaling with the size of the dataset. This poses a computational challenge when solving the finite-horizon problem~\eqref{eq:DR-RLMPC:main} with $\overline{\XX}$ set as $\XWass^{j-1}$ given in~\eqref{eq:safe-1}, especially since the terminal constraint in~\eqref{eq:DR-RLMPC:main} is equivalent to a mixed-integer one. To alleviate this roadblock, we propose the following two alternatives. 

\subsection{Safety set $\XcluWass^{j-1}$}
As explained above, using each sample in defining the ambiguity set increases the computational burden in the reformulated problem. Hence, we use clustering following the ideas outlined in~\cite{DL-SM:2020} and~\cite{FF-PJG:2021}. 
In particular, we assume that the designer selects the number of clusters $\Ncl \in \naturalnumbers_{\ge 1}$ based on the trade-off between accuracy and computational cost. Given $\Ncl$ and samples $\{\what_1, \dots, \what_{N_{j-1}}\}$, we determine the centers of clusters $\{w^{\mathrm{clu},j-1}_\ell\}_{\ell =1}^{\Ncl} \subset \WW$ and an association map $\idmap:\{1,\dots,N_{j-1}\} \to \{1,\dots,\Ncl\}$ such that each sample $\what_i$ is associated with some cluster given by $\idmap(i)$. Then, we form the clustered empirical distribution as 
\begin{align*}
	\Pb^{\mathrm{clu},j-1} := \textstyle \sum_{\ell =1}^{\Ncl} p_\ell \delta_{w^{\mathrm{clu},j-1}_\ell},
\end{align*} 
where for each $\ell \in \{1,\dots,\Ncl\}$, we have
\begin{align*}
	p_\ell = \frac{\abs{\setdef{i \in \{1,\dots,N_{j-1}\}}{\idmap(i) = \ell}} }{N_{j-1}}. 
\end{align*}
That is, the fraction of points associated with the cluster $\ell$. Further, we select the radius of the ambiguity set as $\theta_{\mathrm{clu},j-1} = \theta_{j-1} + \bar{d}$, where
	$\bar{d} = \max_{i \in \{1,\dots,N_{j-1}\}} \norm{\what_i - w_{\idmap(i)}^{\mathrm{clu},j-1}}$.
We then use $\Pb^{\mathrm{clu},j-1}$ and $\theta_{\mathrm{clu},j-1}$ in place of $\Pbcen$ and $\thetacen$, respectively, in Lemma~\ref{le:reform} and define the safety set 
\begin{align}
	\Xb&(\Data^{j-1},\theta_{j-1}) = \XcluWass^{j-1}  \label{eq:safe-2}
	\\
	& := \Big\{x \in \XX | \exists \nu, \gamma_\ell\in\realnonnegative^M, \;\lambda, s_\ell\in \realnonnegative, \;\eta \in \real,\; \notag
		\\
		& \qquad \quad \text{such that } \notag 
		\\
		&  \qquad  \textstyle \lambda \theta_{\mathrm{clu},j-1} - \beta\eta + \sum_{\ell = 1}^{\Ncl} s_\ell \le 0, \notag 
		\\
		&  \qquad  p_\ell^{-1}s_\ell - \eta \ge \dmin - \big([ACx-b]^\top\nu \notag 
		\\
		&  \qquad  \, \, -[(A^\top\!\nu\!-\! H^\top\!\gamma_\ell)^\top\! w_\ell^{\mathrm{clu},j-1}\!+\!\gamma_\ell^\top\! h]\big), \forall \ell\in [\Ncl], \notag 
		\\
		& \qquad  \norm{A^\top\nu\!-\! H^\top\gamma_\ell}\!\le\!\lambda, \, \norm{A^\top\nu}\!\le\!1,  \forall \ell\in[\Ncl] \Big\}. \notag  
\end{align}

Note that the advantage of the above safety set lies in the fact that the set is defined with fewer constraints and this number only depends on the number of clusters $\Ncl$ instead of $N_{j-1}$. While we have reduced the number of points in the support of the center of the ambiguity set, we have increased the radius. In the process we have retained the guarantee: 
\begin{lemma}\longthmtitle{Confidence guarantee of $\Pb$ contained in clustered data ambiguity set}
	Consider any iteration $j$ and suppose Assumption~\ref{as:conf} holds. 
	Then, we have
	\begin{align*}
		\operatorname{Prob}[\Pb \in \mathcal{B}(\Pb^{\mathrm{clu},j-1},\theta_{\mathrm{clu},j-1})] \ge \zeta.
	\end{align*}
\end{lemma}
The proof follows directly from that of~\cite[Lemma VIII.2]{DL-SM:2020}. Owing to the above result, by ensuring $\Xsafeiter^{j-1} = \XcluWass^{j-1}$ for iteration $j$, we obtain a trajectory that satisfies the risk constraint~\eqref{eq:IHOCP-d} with probability $\zeta$. 
Finally, we comment that there are several ways of obtaining the clustered distribution $\Pbclu$ from the dataset $\Data^{j-1}$. In our numerical example, we use $K$-means clustering. Another interesting approach is to use the I-Cover algorithm from~\cite{DL-SM:2020}. However, in this case, the number of clusters grows as the number of samples increases.

\subsection{Safety set $\XinnWass^{j-1}$}

Here we take a different approach as compared to the earlier methods. Instead of imposing distributionally robust constraints in the finite-horizon problem of the MPC, we create an inner estimate of the set $\XX^{j-1}_{\mathrm{DR}}$ (see~\eqref{eq:XWass}) at the beginning of the $j$-th iteration. We then force our trajectory to lie in this inner-estimated set during the MPC implementation. To achieve this, we present the following result that serves as a tool for constructing the inner estimation of $\XX^{j-1}_{\mathrm{DR}}$.
\begin{proposition}\longthmtitle{Inner-estimating $\XX^{j-1}_{\mathrm{DR}}$}\label{pr:inner}
	Consider the $j$-th iteration and the ambiguity set $\mathcal{B}(\Pbhat_{N_{j-1}},\theta_{j-1})$ given in~\eqref{eq:B-Wass}. Define for each $m \in [M]$ the set   
	\begin{align*}
		\XXinner{j-1}_m & := \setdef{x \in \XX}{\sup_{\Qb \in \mathcal{B}(\Pbhat_{N_{j-1}},\theta_{j-1})} \CVaR_\beta^\Qb [b_m + A_m^\top w 
			\\
			& \qquad  \qquad \qquad \qquad + \dmin - A_m^\top C x] \le 0},
	\end{align*}
	where we recall that $A_m$ and $b_m$ are the $m$-th row of $A$ and $m$-th component of $b$, respectively. Let $\XXinner{j-1} := \cup_{m=1}^M \XXinner{j-1}_m$. If $x \in \XXinner{j-1}$, then 
	\begin{align*}
		\sup_{\Qb \in \mathcal{B}(\Pbhat_{N_{j-1}},\theta_{j-1})} \CVaR_\beta^\Qb [ g(x,w) ] \le 0,
	\end{align*}
	where $g$ is given in~\eqref{eq:g-func}.
\end{proposition}
\begin{proof}
	Note that for any $y \in \OO_w$ we have $Ay \le b+Aw$. We will establish the result by showing that for any $x \in \XX$ and $w \in \WW$, if  
	\begin{align}\label{eq:b-ineq}
		b_m + A_m^\top w + \dmin  - A_m^\top C x \le 0,
	\end{align}
	for some $m \in \{1,\dots,M\}$, then $g(x,w) \le 0$. To this end, note that
	\begin{align*}
		b_m  + A_m^\top w &+ \dmin  - A_m^\top C x 
		\\
		& = b_m + A_m^\top w + \max_{\norm{f} \le \dmin} f^\top A_m - A_m^\top  C x 
		\\
		& \overset{(a)}{\ge} \max_{y \in \OO_w} A_m^\top y + \max_{\norm{f} \le \dmin} f^\top A_m - A_m^\top  C x 
		\\
		& = \max_{y \in \OO_w + \dmin \mathcal{B}} A_m^\top y - A_m^\top  C x,
	\end{align*}
	where $\mathcal{B}$ is the unit ball centered at the origin, the first equality is due to $\norm{A_m} = 1$, and (a) follows from the fact that $A_m^\top y \le b_m + A_m^\top w$ for any $y \in \OO$. From the above inequality, if~\eqref{eq:b-ineq} is satisfied, then
		$\max_{y \in \OO_w + \dmin \mathcal{B}} A_m^\top y - A_m^\top C x \le 0$.
	This implies $C x \not \in \mathrm{int}( \OO_w + \dmin \mathcal{B} )$, where $\mathrm{int}(\cdot)$ represents the interior. That is, 
	$\max_{y \in \OO_w + \dmin \mathcal{B}} \norm{C x - y} \ge 0$ and so, $\dist(Cx,\OO_w) \ge \dmin$. This completes the proof.
\end{proof}
Roughly speaking, in the above result we have constructed a safe region $\XXinner{j-1}$ as the union of sets $\XXinner{j-1}_m$ that are themselves intersection of half-spaces and the set $\XX$. The outward normal defining the half-space is chosen to be the same that defines the occupancy of the obstacle. For example, in two dimensions, if the obstacle is a square with normal directions aligning with the axes, then the region $\XXinner{j-1}$ will turn out to be the whole space $C \XX$ except for the set of points that belong to a rectangular region. This fact will become more clear in our simulation section. 

In light of Proposition~\ref{pr:inner}, given the dataset $\Data^{j-1}$ and radius $\theta_{j-1}$, we set 
\begin{align}\label{eq:safe-3}
	\Xb(\Data^{j-1},\theta_{j-1}) = \XinnWass^{j-1} = \XXinner{j-1}
\end{align}
in Algorithm~\ref{ag:DR-iteration}. Then, the trajectory generated in the $j$-th iteration is $\Pb$-risk safe with probability $\zeta$. We next comment on the procedure of determining the constraints defining the set $\XinnWass^{j-1}$. We first compute
\begin{align*}
	\binner{j-1}_m = \sup_{\Qb \in \mathcal{B}(\Pbhat_{N_{j-1}},\theta_{j-1})} \CVaR_\beta^\Qb [A^\top_m w]
\end{align*} 
for all $m \in \{1,\dots,M\}$. Consequently, we set %
\begin{align*}
	\XinnWass^{j-1} &= \cup_{m=1}^M \setdef{x \in \XX}{
		\\
		& \qquad b_m + \binner{j-1}_m + \dmin - A_m^\top C x \le 0}.
\end{align*}
While the above set has a simple form, we still require computing the scalar values $\{\binner{j-1}_m\}$. To this end, we provide the following finite-dimensional optimization problem:
\begin{align*}
	&\binner{j-1}_m = 
	\\
	&\begin{cases} \nonumber\inf  & \, \, \lambda\theta_{j-1} \! - \! \beta\eta + \sum_{\ell = 1}^L s_\ell
		\\
		\st   
		& \, \, p_\ell^{-1}s_\ell \! - \! \eta \! \ge \! (A_m^\top \!- \! H_m^\top\xi_\ell)^\top \what_\ell + \xi_\ell h_m,  \forall  \ell \! \in \! [N_{j-1}],
		\\
		& \, \, \norm{A_m^\top-H_m^\top\xi_\ell} \le \lambda, \qquad \forall \ell \in [N_{j-1}],
		\\
		& \, \, \lambda, s_\ell, \xi_\ell \in \realnonnegative, \;\eta \in \real, \quad\! \forall \ell \in [N_{j-1}]. %
	\end{cases}
\end{align*}

We next conclude this section with the guarantees that our methods enjoy. The first result states the safety and recursive feasibility and then we present the asymptotic convergence. The proofs are omitted as they are similar to those given in~\cite{AZ-AC:22-ecc} for analogous statements.

\begin{proposition}\longthmtitle{Safety and recursive feasibility of Algorithm~\ref{ag:DR-iteration}}
	 	Let Assumption~\ref{as:conf} and~\ref{assump:safeset} hold. Then, for each of the safety sets represented by maps~\eqref{eq:safe-1},~\eqref{eq:safe-2}, and~\eqref{eq:safe-3}, at each iteration $j\ge1$ and time step $t \ge 0$, the finite-horizon problem~\eqref{eq:DR-RLMPC:main} with $x = x_t^j$, $\SSo = \SSs^{j-1}$, and $\overline{\XX} = \mathbb{X}(\Data^{j-1},\theta_{j-1})$ solved in Algorithm~\ref{ag:DR-iteration} is feasible. Further, each point in the generated trajectory $(\xtraj^j,\utraj^j)$ satisfies the risk-constraint~\eqref{eq:IHOCP-d} with probability $\zeta$. 
\end{proposition}
\begin{proposition}\longthmtitle{Asymptotic convergence of Algorithm~\ref{ag:DR-iteration}}
	Let Assumption~\ref{assump:safeset} hold. Then, for each of the safety sets represented by maps~\eqref{eq:safe-1},~\eqref{eq:safe-2}, and~\eqref{eq:safe-3} and each iteration $j \ge 1$ of Algorithm~\ref{ag:DR-iteration}, the trajectory $(\xtraj^j,\utraj^j)$ generated by $\drmpc$ satisfies $x^j_t \to x_F$ as $t \to \infty$.
\end{proposition}

\begin{figure}%
	\centering
	\begin{subfigure}[b]{0.79\columnwidth}
		\centering
		\includegraphics[width=\linewidth]{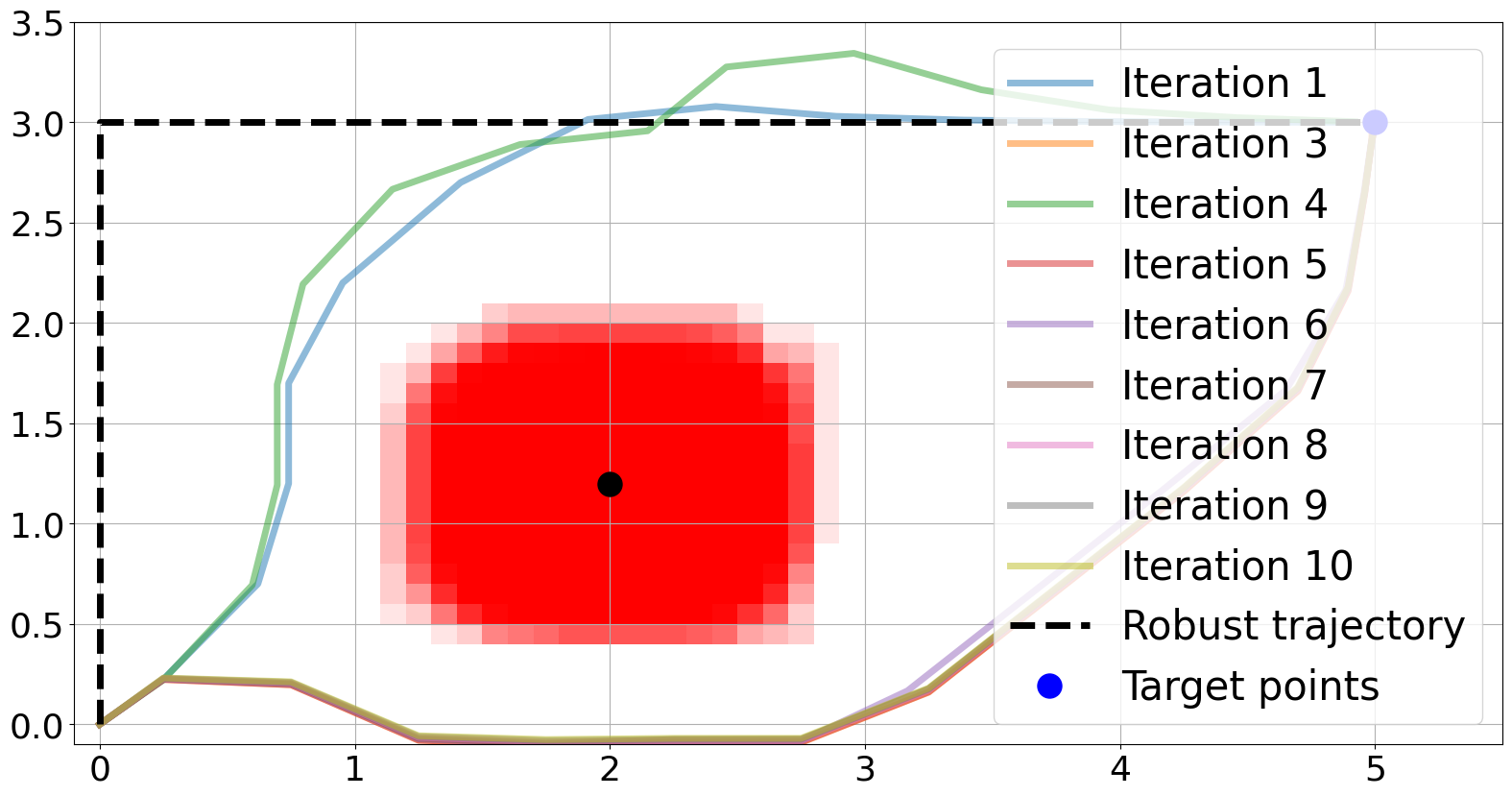}
		\caption[Network2]%
		{{\footnotesize Safety set $\XWass$}}    
		\label{fig:trajs_a}
	\end{subfigure}
	\\
	\begin{subfigure}[b]{0.79\columnwidth}  
		\centering 
		\includegraphics[width=\linewidth]{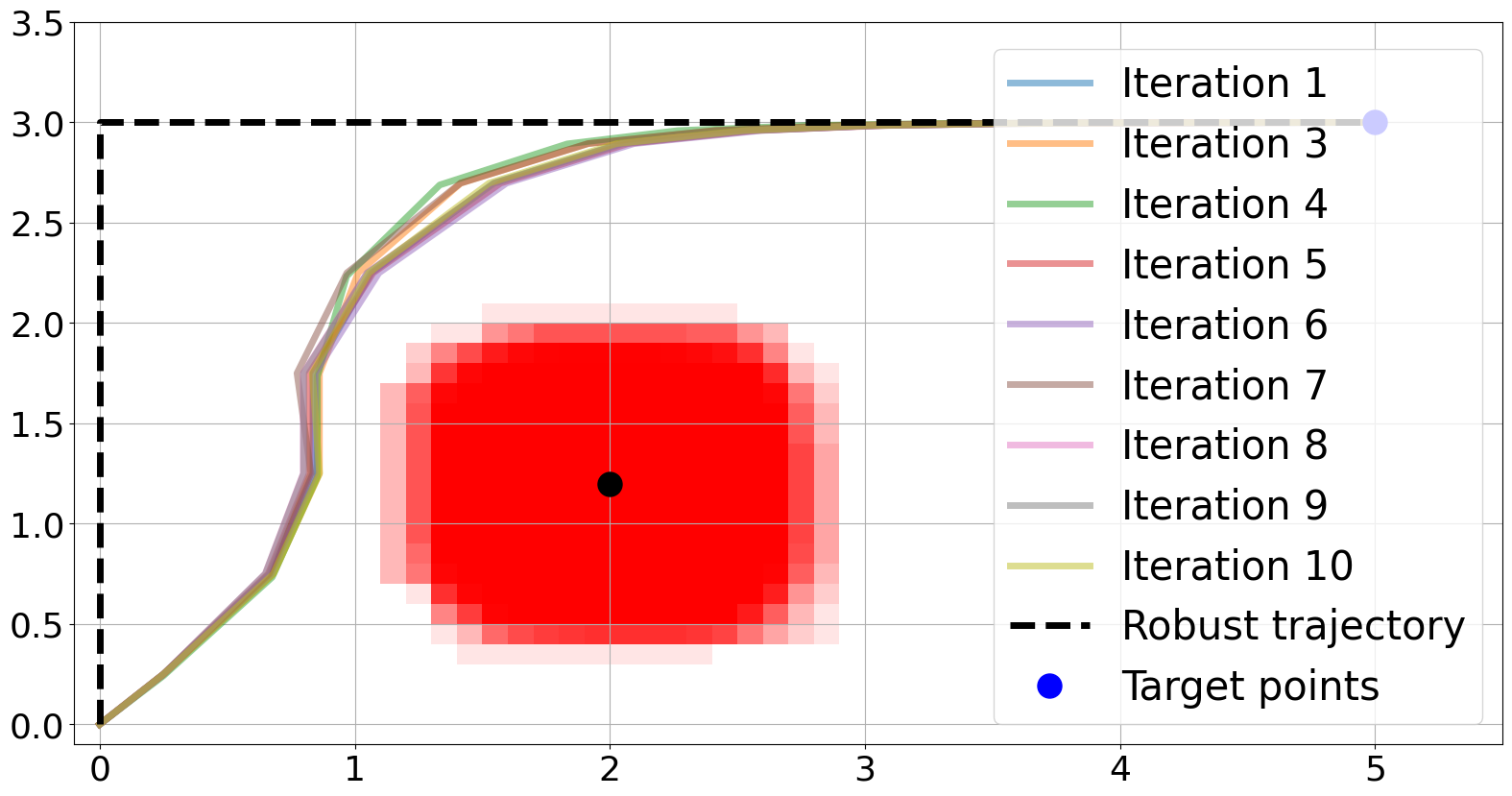}
		\caption[]%
		{{\footnotesize Safety set $\XcluWass$}}
		\label{fig:trajs_b}
	\end{subfigure}
\\
	\begin{subfigure}[b]{0.79\columnwidth}   
		\centering 
		\includegraphics[width=\linewidth]{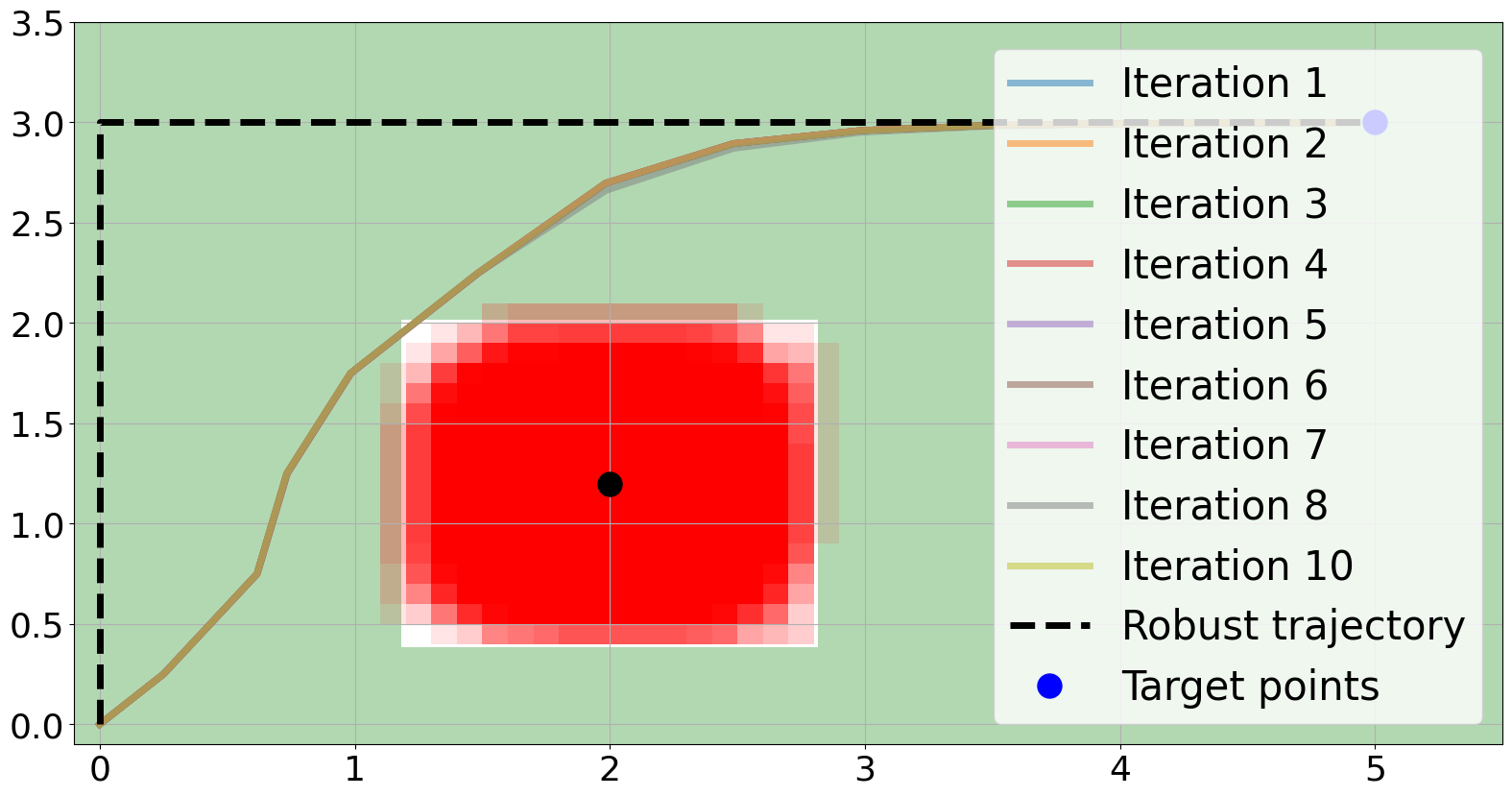}
		\caption[]%
		{{\footnotesize Safety set $\XX^{\mathrm{inn}}$}}    
		\label{fig:trajs_c}
	\end{subfigure}
	\caption{\footnotesize Plots depicting the application of Algorithm~\ref{ag:DR-iteration} with the presented approximations for a path planning task in presence of an uncertain obstacle (see Section~\ref{sec:sims} for details). The dashed black line represents the initial robust trajectory and the green area in Figure~\ref{fig:trajs_c} shows $\XX^{\mathrm{inn}}$ in the last iteration. The obstacle's occupancy is depicted by the shaded red heat map.} %
\vspace*{-3ex}
\label{fig:trajs}
\end{figure}
\begin{figure*}%
\centering
\begin{subfigure}[b]{0.31\linewidth}
	\centering
	\includegraphics[width=\linewidth]{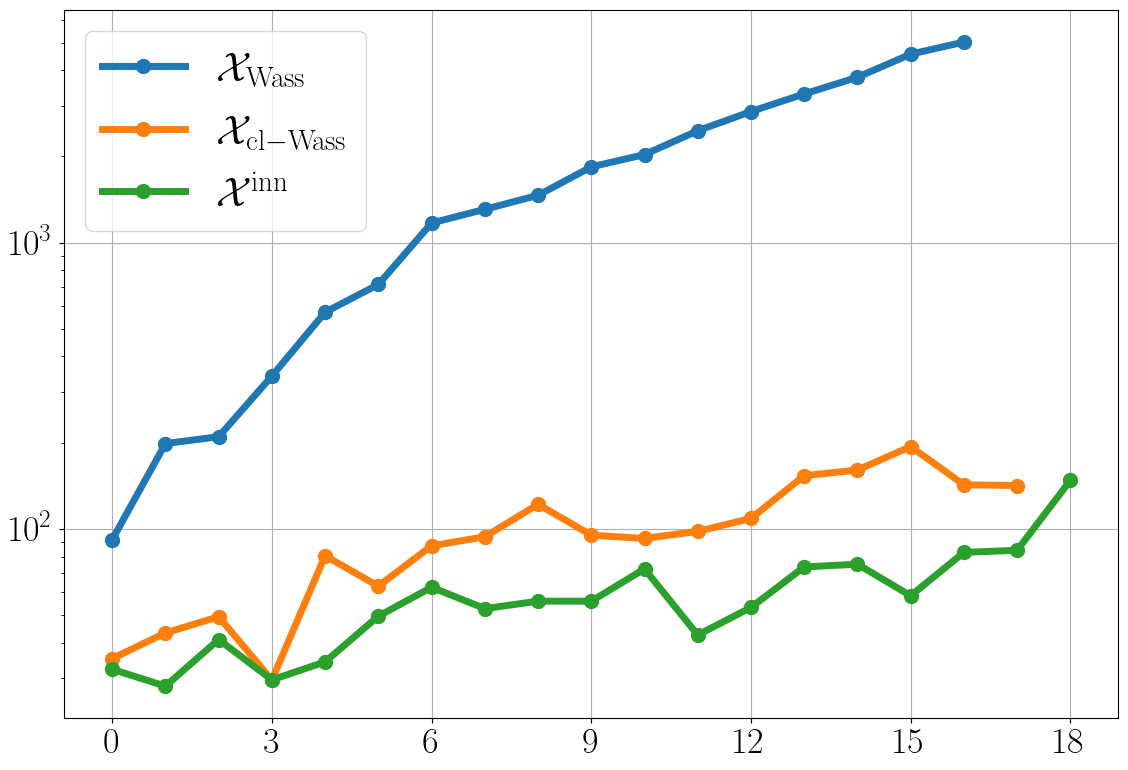}
	\caption[Network2]%
	{{\footnotesize Elapsed time in each iteration (sec)}}    
	\label{fig:iter}
\end{subfigure}
\quad
\begin{subfigure}[b]{0.31\linewidth}  
	\centering 
	\includegraphics[width=\linewidth]{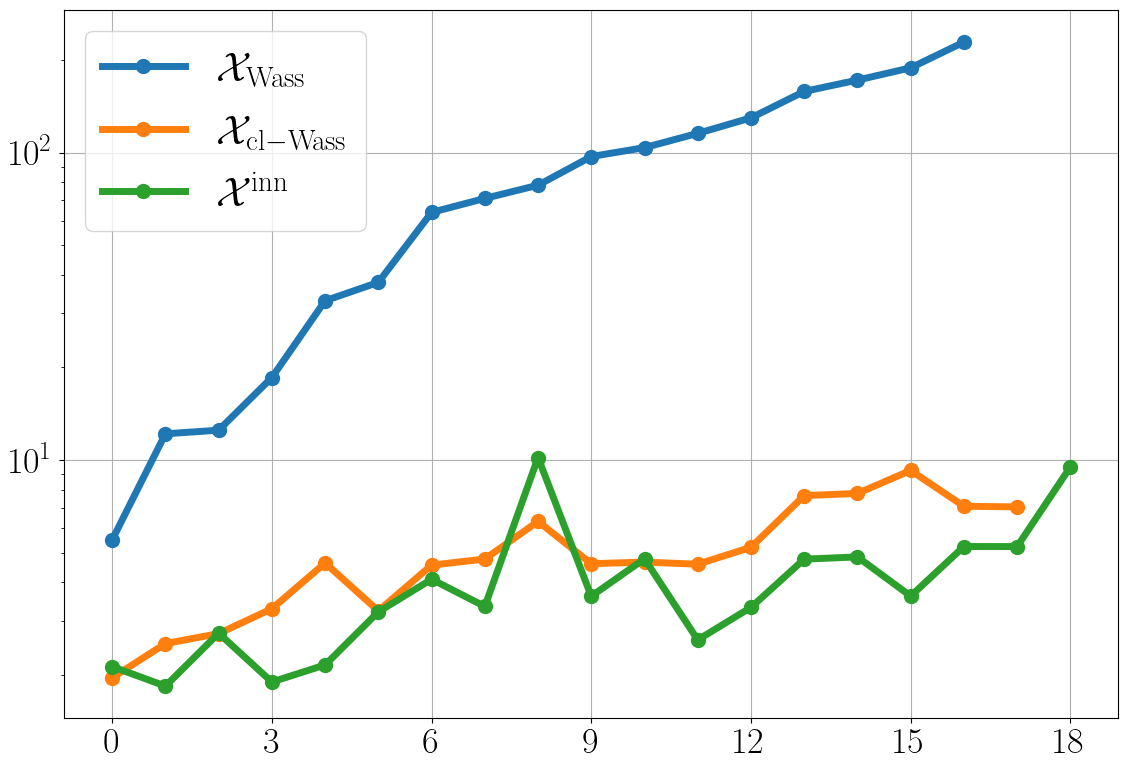}
	\caption[]%
	{{\footnotesize Steps average elapsed time (sec)}}
	\label{fig:step}
\end{subfigure}
\quad
\begin{subfigure}[b]{0.31\linewidth}   
	\centering 
	\includegraphics[width=\linewidth]{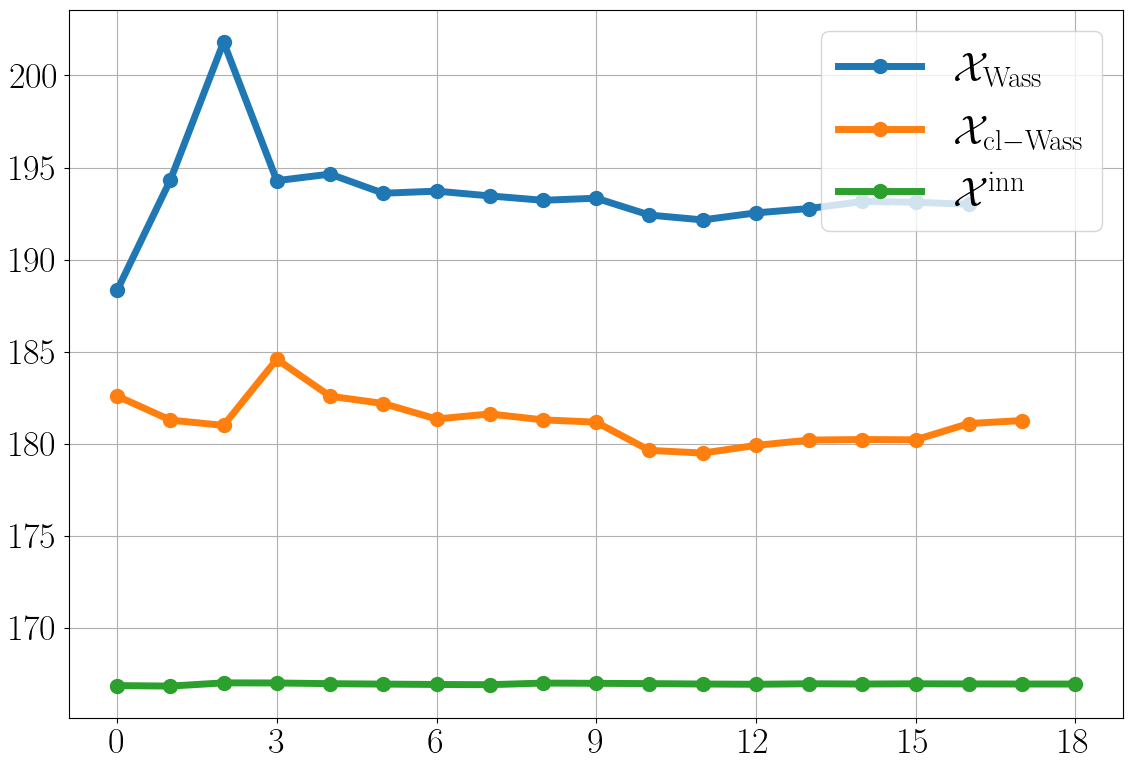}
	\caption[]%
	{{\footnotesize Cost of each iteration}}    
	\label{fig:cost}
\end{subfigure}
\caption{\footnotesize The comparison of three presented approximation approaches in Algorithm~\ref{ag:DR-iteration} in terms of the computational time and cost efficiency. Plots illustrate the elapsed time in each iteration, elapsed time per step in each iteration, and the cost of each iteration over the progress of the algorithm.} 
\vspace*{-3ex}
\label{fig:effic}
\end{figure*}

\section{Simulation}\label{sec:sims}
In this section, we consider a motion planning problem in the presence of a randomly moving obstacle to compare the performance and efficiency of the presented approaches. In this problem, a circular mobile robot is navigating in a 2-D environment. We assume that at each time step, the position of the obstacle is observable and as the iterations progress, more information is revealed. 

\subsubsection{Setup}
The deterministic linear dynamics of the mobile robot is represented as:
\begin{align*}
	x_{t+1} & = \begin{bmatrix}
		1 & 0 & 1 & 0 \\
		0 & 1 & 0 & 1 \\
		0 & 0 & 1 & 0 \\
		0 & 0 & 0 & 1 
	\end{bmatrix}x_t + \begin{bmatrix}
		0 & 0 \\
		0 & 0 \\
		1 & 0 \\
		0 & 1 \\
	\end{bmatrix}u_t,
\end{align*}
where the state vector $x_t$ consists of the position and velocity and the input vector $u_t$ represents the acceleration. At each iteration, the agent starts at $x_s = [0, 0, 0, 0]$ and aims to reach the target at $x_F = [5, 3, 0, 0]$, while a square obstacle is moving randomly with the position of its center being $o_t = [2, 1.2] + w_t$. The distribution of $w_t\in\real^2$ is the product distribution, where for each axis, the component distribution is a truncated zero-mean normal distribution with support $[-0.45, 0.45]$ and variance $\sigma\!=\!0.15$. Before starting the first iteration, the observations set is initialized with 15 i.i.d samples. The stage cost $r(x, u)$ is in quadratic form and given as ${r(x_t, u_t)=(x_F-x_t)^{\top}Q(x_F-x_t) + u_t^{\top}Ru_t}$, where ${Q=\text{diag}(1, 1, 0.01, 0.01)}$ and ${R=\text{diag}(0.01, 0.01)}$. Other parameters are specified in Table~\ref{tb:params}. The optimization problem~\eqref{eq:DR-RLMPC:main} is solved using GEKKO~\cite{LB-DH-RM-JH:2018} on a PC with an Intel Core i7-10610U 2.30-GHz processor and 16-GB RAM. For $\XWass^{j-1}$ and $\XcluWass^{j-1}$, the problem is solved using an IPOPT-based solver and for $\XXinner{j-1}$ an APOPT-based solver is used. Each algorithm is executed for 20 iterations.
\begin{table}
	\centering
	\resizebox{\columnwidth}{!}{
	\begin{tabular}{||l|l||l|l||}
		\hline
		Risk coefficient ($\beta$) & 0.05 & Distance threshold ($\dmin$) & 0.1\\
		Obstacle length  & 1 & Agent radius& 0.2\\
		Horizon length ($K$) & 11 & Ambiguity set radius ($\theta$) & $1e^{-3}$\\
		 Number of iterations & 20 & Number of clusters ($\Ncl$) & 5 \\
		\hline
	\end{tabular}
}
	\caption{Simulation parameters}
	\vspace*{-3ex}
	\label{tb:params}
\end{table}
\subsubsection*{Results}

Trajectories generated using three approximation approaches of the problem~\eqref{eq:DR-RLMPC:main} are depicted in Figure~\ref{fig:trajs}. Resulted trajectories show that considering $\XWass$ as the safety set, provide the algorithm with more freedom to explore the environment. After expanding the safe set in the first iterations, the agent tries to find an efficient trajectory from the lower side of the obstacle. On the other hand, considering safety sets $\XcluWass$ and $\XX^{\mathrm{inn}}$ does not let the agent deviate much from the initial safe set. In Figure~\ref{fig:trajs_c}, the green area that represents $\XX^{\mathrm{inn}}$ includes some areas in which the probability of the obstacle's presence is low. This shows the difference between the presented distributionally robust risk-averse set $\XX^{\mathrm{inn}}$ and a normal robustly safe set. It is notable that the conservatism of the algorithms can be tuned via $\dmin$ and $\beta$ which is not the focus of this experiment.

In Figure~\ref{fig:effic}, we compare the efficiency of the presented approaches in terms of both performance and computational effort. The first two plots show that using safety sets $\XcluWass$ and $\XX^{\mathrm{inn}}$ in problem~\eqref{eq:DR-RLMPC:main} reduces the computational time notably and their difference with the $\XWass$ increase over iterations. In addition, the average elapsed time per step in each iteration indicates that the clustering approach and the inner approximation of the feasible set are comparable in terms of real-time decision-making, however, the inner approximation approach is slightly faster. Furthermore, bilinearity of constraints in $\XWass$ and $\XcluWass$ limits the solver to be able to guarantee optimality. As a result, involving $\XX^{\mathrm{inn}}$ leads to better trajectories in terms of cost efficiency.

\section{Conclusions}
We have considered a risk-constrained optimal control problem for motion planning and explored using iterative DR MPC method as a solution strategy. Considering Wasserstein ambiguity sets in the DR MPC routine, we have formulated two approximations of the optimization problem driving the MPC. We have shown the approximations to be computationally efficient and resulting in safe trajectories. We have illustrated the strength of our methods via a numerical example. Future work includes further bringing down the computational costs by considering reachable sets and exploring distributed implementation 
for multi-robot setup.

\bibliographystyle{ieeetr}

\end{document}